\documentclass{article}

\usepackage{a4wide}

\usepackage[english]{babel}

\usepackage[dvipsnames]{xcolor}
\definecolor{myred}{HTML}{FF3D3D}
\definecolor{mycyan}{HTML}{0474BE}

\usepackage{hyperref}

\usepackage{amsmath,amssymb,amsthm}

\usepackage{booktabs}

\usepackage{enumerate}

\usepackage{units}

\usepackage{subfig}

\numberwithin{equation}{section}

\usepackage{cleveref}

\usepackage{mathtools}

\usepackage[obeyFinal,
                color       = mycyan!40!white,
                bordercolor = black,
                textwidth   = 25mm,
                figwidth    = 0.99\linewidth]{todonotes}

\usepackage{units}
\newcommand{\R}{\mathbb{R}}

\newcommand{\Z}{\mathbb{Z}}

\newcommand{\sign}{\operatorname{sign}}

\newcommand{\Lone}{\rm{L}^1}

\newcommand{\TV}{\operatorname{TV}}
\newcommand{\BV}{\mathrm{BV}}

\newcommand{\Dx}{{\Delta x}}
\newcommand{\Dt}{{\Delta t}}
\newcommand{\hf}{{\unitfrac{1}{2}}}

\newcommand*\diff{\mathop{}\!\mathrm{d}}

\usepackage{bbm}

\usepackage{tikz}
\usepackage{pgfplots}
\pgfplotsset{width=.37\textwidth,compat=1.12}

\definecolor{skyblue1}{rgb}{0.447,0.624,0.812}
\definecolor{plum1}{rgb}{0.678,0.498,0.659}
\definecolor{scarletred1}{rgb}{0.937,0.161,0.161}

\definecolor{myblue}{HTML}{1e77b4}
\definecolor{myorange}{HTML}{ff7f0f}

\definecolor{mycolor}{rgb}{0.122, 0.435, 0.698}

\usepackage{siunitx}
\usepackage{booktabs}

\newtheorem{theorem}{Theorem}[section]
\newtheorem{lemma}[theorem]{Lemma}
\newtheorem{definition}[theorem]{Definition}
\newtheorem{proposition}[theorem]{Proposition}
\newtheorem{remark}[theorem]{Remark}
\newtheorem*{maintheorem}{Main theorem}
\newtheorem{corollary}[theorem]{Corollary}

\allowdisplaybreaks

\title{Flux-stability for conservation laws with discontinuous flux and convergence rates of the front tracking method}
\author{\textsc{Adrian\,M. Ruf}\thanks{Seminar for Applied Mathematics, ETH Z\"urich, Switzerland (\texttt{adrian.ruf@sam.math.ethz.ch})}}
\date{\today}

\begin{document}

\maketitle

\begin{abstract}
	We prove that adapted entropy solutions of scalar conservation laws with discontinuous flux are stable with respect to changes in the flux under the assumption that the flux is strictly monotone in $u$ and the spatial dependency is piecewise constant with finitely many discontinuities. We use this stability result to prove a convergence rate for the front tracking method -- a numerical method which is widely used in the field of conservation laws with discontinuous flux. To the best of our knowledge, both of these results are the first of their kind in the literature on conservation laws with discontinuous flux.
	We also present numerical experiments verifying the convergence rate results and comparing numerical solutions computed with the front tracking method to finite volume approximations.
\end{abstract}

\paragraph{Key words.} hyperbolic conservation laws, discontinuous flux, stability, front tracking, convergence rate

\paragraph{AMS subject classification.} 35L65 , 35R05, 35B35, 65M12


\section{Introduction}
We consider scalar conservation laws with discontinuous flux of the form
\begin{gather}
	\begin{aligned}
		u_t + f(k(x),u)_x = 0,& &&(x,t)\in\R\times(0,T)\\
		u(x,0) = u_0(x),& &&x\in\R.
	\end{aligned}
\label{conservation law}
\end{gather}
Here, the flux $f$ is smooth in $k$ and $u$, but may have a \emph{discontinuous spatial dependency} through the coefficient $k$.

\emph{The aim of this paper is to show that entropy solutions of~\eqref{conservation law} are stable with respect to changes in $k$ and $f$}. We use this stability result to show that the front tracking method -- an important tool to show existence of solutions for~\eqref{conservation law} --  has a \emph{first order convergence rate}.
Both of these results constitute generalizations of the seminal theory developed by Lucier \cite{lucier1986moving} for conservation laws without spatial dependency in the flux.
\begin{maintheorem}
  Let $f\in\mathcal{C}^2$ be strictly monotone in $u$ in the sense that $f_u\geq \alpha>0$, $k$ piecewise constant with finitely many discontinuities, and $u_0\in(\mathrm{L}^1\cap\BV)(\R)$. Then entropy solutions of~\eqref{conservation law} are $\mathrm{L}^1$-stable with respect to changes in the flux, the discontinuous coefficient $k$, and the initial datum. Moreover, the front tracking algorithm has a first-order convergence rate.
\end{maintheorem}
The precise theorem is stated in \Cref{sec: main stability result,sec: Convergence rate front tracking}.
Our proof uses Kuznetsov-type lemmata on the subdomains between neighboring discontinuities of $k$ and a novel spatial total variation bound of the flux to derive stability estimates on each subdomain.

\subsection{Background on conservation laws with discontinuous flux}
Equations of type~\eqref{conservation law} arise in a number of areas of application including vehicle traffic flow in the presence of abruptly varying road conditions (see \cite{lighthill1955kinematic}), polymer flooding in oil recovery (see \cite{Shen:2017aa}), two-phase flow through heterogeneous porous media (see \cite{gimse1992solution,risebro1991front}), and sedimentation processes (see \cite{diehl1996conservation,burger2003front}).

When the spatial dependency is smooth, well-posedness of the conservation law~\eqref{conservation law} is well-known due to the seminal results of Kru\v{z}kov \cite{kruvzkov1970first}. Uniqueness follows from the so-called Kru\v{z}kov entropy inequality
\begin{equation}
	\partial_t |u - c| + \partial_x \left( \sign(u-c)(f(k(x),u) - f(k(x),c)) \right) + \sign(u-c)\partial_x f(k(x),c) \leq 0
	\label{eqn: Kruzkov entropy condition}
\end{equation}
which is to be satisfied in the distributional sense and for all $c\in\R$. However, when spatial flux discontinuities are present,~\eqref{eqn: Kruzkov entropy condition} no longer makes sense due to the term $\sign(u-c)\partial_x f(k(x),c)$. This difficulty is usually overcome by requiring that~\eqref{eqn: Kruzkov entropy condition} holds away from the spatial flux discontinuities, and imposing suitable jump conditions along the spatial interfaces (in addition to the Rankine--Hugoniot condition) \cite{gimse1991riemann,gimse1992solution,diehl1996conservation,klingenberg1995convex,klausen1999stability,Towers1,Towers2,karslen2003l1,karlsen2004convergence,adimurthi2005optimal,andreianov2011theory}.

In the last two decades many different selection criteria to that effect were proposed. Most theories are restricted to fluxes with separated variables, i.e., fluxes of the type $f(k(x),u)=k(x)f(u)$ \cite{klingenberg1995convex,Towers1,Towers2,wen2008convergence}, where $k$ has finitely many discontinuities, usually only one.

In \cite{gimse1991riemann,gimse1992solution}, Gimse and Risebro used the minimal jump entropy condition at the flux interfaces to obtain uniqueness for the Riemann problem. Diehl imposed a different interface condition to show uniqueness for the Riemann problem which he termed $\Gamma$ condition \cite{diehl1996conservation}. In \cite{klingenberg1995convex}, Klingenberg and Risebro showed that the minimal jump entropy condition implies the wave entropy condition which is sufficient to show uniqueness for the Cauchy problem.
These results were later extended to much more general fluxes by Panov~\cite{panov2010existence}.
Additionally, Aae Klausen and Risebro showed that wave entropy solutions are not only unique, but $\mathrm{L}^1$-stable with respect to changes in the discontinuous coefficient $k$ and the initial datum \cite{klausen1999stability}.
In~\cite{Towers1}, Towers established uniqueness results using certain geometric entropy conditions at the flux interfaces assuming that the solution is piecewise $\mathcal{C}^1$.
Karlsen, Risebro, and Towers employed the crossing condition at the interfaces to obtain uniqueness for~\eqref{conservation law} (including a degenerate parabolic term) \cite{karslen2003l1}.
Adimurthi, Mishra and Veerappa Gowda \cite{adimurthi2005optimal} showed that~\eqref{conservation law} (with just a single discontinuity in $k$) admits many $\mathrm{L}^1$-contractive semigroups of entropy solutions, one for each so-called connection $(A,B)$ (see also~\cite{Sid2005,ADIMURTHI2007310} and references therein).

A different uniqueness theory based on so-called adapted entropies which does not require any additional interface entropy conditions was introduced by Baiti and Jenssen~\cite{BAITI1997161} and further developed by Audusse and Perthame~\cite{audusse2005uniqueness} for fluxes satisfying $f_u\geq \alpha>0$ which is also the setting of the present paper. In the adapted entropy framework the usual Kru\v{z}kov entropies $(|u-c|)_{c\in\R}$ are replaced by the adapted entropies $(|u-c_p(x)|)_{p\in\R}$ where $c_p(x)$ is the unique solution of
\begin{equation*}
  f(k(x),c_p(x)) = p\qquad\text{for all }x \in\R.
\end{equation*}
If $c_p(x)$ takes the place of the constant $c$ in the Kru\v{z}kov entropy condition~\eqref{eqn: Kruzkov entropy condition}, the problematic term $\sign(u-c_p(x))\partial_x f(k(x),c_p(x))$ vanishes (cf. \Cref{def: entropy solution} below). This allows the definition of entropy solutions without imposing additional interface conditions and, in particular, without requiring the existence of traces. In fact, instead of assuming that the flux in~\eqref{conservation law} has the form $f(k(x),u)$ for $f\in\mathcal{C}^2$ \cite{BAITI1997161} the adapted entropy theory even allows for more general fluxes of the form $f(x,u)$ where $f$ is only Lipschitz in $u$ and may have infinitely many spatial discontinuities~\cite{audusse2005uniqueness}. Existence of adapted entropy solutions for general fluxes of this form was shown by Piccoli and Tournus~\cite{Piccoli/Tournus} assuming that $f$ is concave in $u$. Towers recently extended this existence result to non-concave fluxes~\cite{TOWERS20205754}.
Furthermore, we want to mention two other recent works regarding adapted entropy solutions of~\eqref{conservation law} concerning a multilevel Monte Carlo framework where the model parameters $(k,f,u_0)$ are subject to uncertainty~\cite{badwaik2019multilevel} and concerning Bayesian inverse problems for conservation laws with discontinuous flux~\cite{MishraBayesian}.

The work~\cite{andreianov2011theory} by Andreianov, Karlsen, and Risebro contains a thorough review of selection criteria available in the literature on conservation laws with discontinuous flux for the so-called `two flux' case -- given by
\begin{equation*}
  u_t + (H(x)f(u) + (1-H(x))g(u))_x = 0
\end{equation*}
where $H$ is the Heaviside function and $f$ and $g$ are Lipschitz continuous. Each existing selection criterion that leads to an $\mathrm{L}^1$-contractivity property is associated to a germ which underlies these conditions \cite{andreianov2011theory}.

\subsection{Background on numerical methods for conservation laws with discontinuous flux}
Due to the various interface conditions that are used to obtain uniqueness for~\eqref{conservation law}, the front tracking method which explicitly deals with the interfaces is a widely used tool to show existence of solutions \cite{gimse1991riemann,gimse1992solution,gimse1993conservation,klingenberg1995convex,klingenberg2001stability,burger2003front,coclite2005conservation,holden2015front}. Also in the adapted entropy framework the front tracking method has been employed successfully to prove existence~\cite{BAITI1997161,Piccoli/Tournus}. However, to the best of our knowledge, no convergence rate results for the front tracking method are currently available in the literature.

In the absence of a spatial dependency of the flux, Lucier~\cite{lucier1986moving} showed that the front tracking method has a first-order convergence rate in $\mathrm{L}^1$. To that end, Lucier first proved that entropy solutions are Lipschitz-stable with respect to changes in the flux. Since numerical solutions computed with the front tracking method are entropy solutions of an approximate conservation law whose flux is a piecewise linear approximation, the convergence rate follows from the flux stability result. We will employ the same general strategy in the present paper between neighboring discontinuities of $k$. Convergence rates of the front tracking method for conservation laws without spatial dependency measured in the $p$-Wasserstein distances have been established by Solem~\cite{solem2018convergence}.

Other than the front tracking method, many authors have proposed and analyzed various finite volume methods \cite{Towers1,Towers2,karlsen2002upwind,karlsen2004convergence,adimurthi2005optimal,Sid2005,ADIMURTHI2007310,wen2008convergence,burger2009engquist,karlsen2017convergence}. In particular, some recent contributions also design finite volume methods in the framework of adapted entropy solutions \cite{BadwaikRufconvergence,TOWERS20205754,ghoshal2020convergence,badwaik2019multilevel}.

Again, in the absence of a spatial dependency of the flux, convergence rates of monotone finite volume methods are well-understood in both $\mathrm{L}^1$~\cite{kuznetsov76,sabac} and in the Wasserstein distance~\cite{nessyconv92,nessyconv,Ruf2019}.
Results regarding convergence rates for finite volume methods for~\eqref{conservation law}, however, are sparse. For the linear advection equation where the wave speed has a single discontinuity, Jin and Wen~\cite{wen2008convergence} proved a convergence rate of $\hf$ for the upwind scheme. For general nonlinear fluxes with $k$ piecewise constant, a recent contribution by Badwaik and the author \cite{BadwaikRufconvergence} showed an optimal convergence rate of $\hf$ for upwind-type finite volume methods
which include the Godunov and Engquist--Osher schemes.

\subsection{Outline of the paper}
We have organized the paper in the following way. In \Cref{sec: Preliminaries}, we will define entropy solutions of~\eqref{conservation law} in the sense of~\cite{BAITI1997161,audusse2005uniqueness} and lay out the main proof strategy of decomposing the problem into finitely many initial(-boundary) value problems on subdomains between two neighboring discontinuities of $k$. \Cref{sec: Preliminaries} also contains a spatial total variation estimate of the flux function which we will leverage in the proof of the main theorem. In \Cref{sec: Stability with one discontinuity}, we consider the case of a single discontinuity in $k$ located at zero and prove stability in $f$ on $\R^-$, $\R^+$, and on the bounded interval $(0,L)$. Those three scenarios will be the cornerstones in showing flux-stability on each subdomain in the general case of arbitrarily located flux discontinuities. \Cref{sec: main stability result} contains the statement and proof of the main flux-stability result where we use the translation invariance of~\eqref{conservation law} to first show stability in the flux $f$ and then, more generally, stability in $f$ and $k$. In \Cref{sec: Convergence rate front tracking} we introduce the front tracking method and apply the flux-stability result to derive a convergence rate. \Cref{sec: numerical experiments} is devoted to numerical experiments verifying our convergence rate result. In addition to that, we compare numerical solutions computed with the front tracking method to finite volume approximations. We summarize the findings of this paper in \Cref{sec: conclusion}.


\section{Preliminaries}\label{sec: Preliminaries}



Throughout this paper, we will assume that the initial datum $u_0$ is integrable, bounded, and of finite total variation, i.e., $u_0 \in(\mathrm{L}^1\cap\BV)(\R)$, and that $f$ is strictly monotone in $u$, i.e., $f_u\geq \alpha>0$, and satisfies $f(k^*,0)=0$ for all $k^*\in\R$. Further, we will denote the discontinuities of $k$ as $\xi_1,\ldots,\xi_N$ and the interval between two adjacent discontinuities as $D_i = (\xi_i,\xi_{i+1})$, $i=0,\ldots, N$. Here, we used the notation $\xi_0=-\infty$ and $\xi_{N+1}=\infty$. Then we can write
\begin{equation*}
  f(k(x),\cdot) =: f^{(i)}(\cdot)\qquad \text{for }x\in D_i.
\end{equation*}
We will consider adapted entropy solutions of~\eqref{conservation law} in the sense of~\cite{BAITI1997161,audusse2005uniqueness} which in the case of piecewise constant $k$ reads as follows.
\begin{definition}[Entropy solution]\label{def: entropy solution}
  We say $u\in\mathcal{C}([0,T];\mathrm{L}^1(\R))\cap\mathrm{L}^\infty((0,T)\times\R)$ is an entropy solution of~\eqref{conservation law} if for all $c\in\R$
  \begin{multline*}
    \sum_{i=0}^N \bigg( \int_0^T \int_{D_i} \left(|u-c_i|\varphi_t + \sign(u-c_i)(f^{(i)}(u)-f^{(i)}(c_i))\varphi_x \right) \diff x\diff t \\
    -\int_{D_i}|u(x,T)-c_i|\varphi(x,T)\diff x +\int_{D_i}|u_0(x)-c_i|\varphi(x,0))\diff x\\
    - \int_0^T \sign(u(\xi_{i+1}-,t)-c_i)(f^{(i)}(u(\xi_{i+1}-,t))-f^{(i)}(c_i))\varphi(\xi_{i+1},t)\diff t \\
    + \int_0^T \sign(u(\xi_i +,t)-c_i)(f^{(i)}(u(\xi_i +,t))-f^{(i)}(c_i))\varphi(\xi_i,t)\diff t \bigg) \geq 0.
  \end{multline*}
  for all nonnegative $\varphi\in\mathcal{C}^\infty_c(\R\times[0,T])$. Here, the $c_i$ are given by $c_0 \coloneqq c$ and
  \begin{equation}
    c_{i+1}=(f^{(i+1)})^{-1}(f^{(i)}(c_i)) \qquad \text{for } i=1,\ldots,N.
    \label{definition c}
  \end{equation}
\end{definition}

\begin{remark}
  Note that, due to the monotonicity of the fluxes $f^{(i)}$, the inverse of $f^{(i)}$ used in~\eqref{definition c} and throughout this paper exists. Moreover, the traces in \Cref{def: entropy solution} are well defined (cf. \cite[Remark 2.3]{andreianov2011theory}).
\end{remark}

The following theorem assures existence and uniqueness of entropy solutions in the present setting.
\begin{theorem}[Existence and uniqueness of entropy solutions \cite{BAITI1997161,TOWERS20205754,BadwaikRufconvergence}]\label{thm: existence and uniqueness of entropy solutions}
  Let $u_0\in(\mathrm{L}^1\cap\BV)(\R)$ and assume that $f$ satisfies $f_u\geq\alpha>0$ and that $k$ is piecewise constant with finitely many discontinuities. Then there exists a unique entropy solution $u$ of~\eqref{conservation law} which satisfies
  \begin{align}
    \|u(\cdot,t)\|_{\mathrm{L}^\infty(\R)} &\leq C \|u_0\|_{\mathrm{L}^\infty(\R)},
    \label{eqn: L infty bound of adapted entropy solutions}\\
    \TV(u(\cdot,t)) &\leq C(\TV(k)+\TV(u_0)),
    \label{eqn: TV bound of adapted entropy solutions}
    \shortintertext{ for all $0\leq t\leq T$ and}
    \TV_{[0,T]} (u(x,\cdot)) &\leq C \TV(u_0)
    \label{eqn: temporal TV bound of adapted entropy solutions}
  \end{align}
  for all $x\in\R$.
\end{theorem}
The existence and uniqueness statement follows from the theory developed by Baiti and Jenssen.
The $\mathrm{L}^\infty$ and $\TV$ bounds in \Cref{thm: existence and uniqueness of entropy solutions} follow from \cite[Thm. 1.4]{TOWERS20205754} and the temporal $\TV$ bound from~\cite[Lem. 4.6]{BadwaikRufconvergence}.


\begin{remark}
  Like for conservation laws without (discontinuous) spatial dependency of the flux, a Rankine--Hugoniot-type argument shows that weak solutions of~\eqref{conservation law} necessarily satisfy the Rankine--Hugoniot condition across all discontinuities~$\xi_i$, i.e.,
  \begin{equation}
    f^{(i-1)}(u(\xi_i-,t)) = f^{(i)}(u(\xi_i+,t))
    \label{continuous Rankine--Hugoniot condition}
  \end{equation}
  holds for almost every $t\in(0,T)$.
\end{remark}

The following observation is at the heart of the proof of the main result. The entropy solution $u$ of~\eqref{conservation law} can be decomposed into a series of entropy solutions of an initial value, respectively initial-boundary value problem on $D_i$.
Specifically, we write $u$ as $u=\sum_{i=0}^N u^{(i)}$ where $u^{(i)}\coloneqq u \mathbbm{1}_{D_i\times [0,T]}$. Then $u^{(0)}$ solves
\begin{gather}
\begin{aligned}
  u^{(0)}_t + f^{(0)}(u^{(0)})_x = 0,& &&(x,t)\in D_0\times(0,T),\\
  u^{(0)}(x,0) = u_0(x),& &&x\in D_0
\end{aligned}
\label{conservation law on D0}
\end{gather}
and $u^{(i)}$ solves
\begin{gather}
\begin{aligned}
  u^{(i)}_t + f^{(i)}(u^{(i)})_x = 0,& &&(x,t)\in D_i\times(0,T),\\
  u^{(i)}(x,0) = u_0(x),& &&x\in D_i,\\
  u^{(i)}(\xi_i +,t) = (f^{(i)})^{-1}\left(  f^{(i-1)}(u^{(i-1)}(\xi_i -,t)) \right),& &&t\in(0,T)
\end{aligned}
\label{conservation law on Di}
\end{gather}
for $i=1,\ldots,N$ (cf. Definitions~\ref{def: entropy solution on R-} and~\ref{def: entropy solution on R+} below).
Note that the boundary condition on the domain $D_i$, $i=1,\ldots,N$, given by the last line of \eqref{conservation law on Di} reflects the Rankine--Hugoniot condition~\eqref{continuous Rankine--Hugoniot condition}.

Conversely, if $u^{(0)}$ is the entropy solution of~\eqref{conservation law on D0} on $D_0$ and $u^{(i)}$ is the entropy solution of~\eqref{conservation law on Di} on $D_i$ for $i=1,\ldots,N$, then the composite function $u\coloneqq \sum_{i=0}^N u^{(i)}$ is the entropy solution of \eqref{conservation law} in the sense of Definition~\ref{def: entropy solution}. This can be seen by adding the entropy inequalities of $u^{(i)}$ and choosing the respective constant in each entropy inequality according to \eqref{definition c}.


To prove our main theorem, we will need the following lemma showing that the flux $f(k(\cdot),u)$ is Lipschitz continuous in space. This consequence of the strict monotonicity of the flux constitutes a generalization of \cite[Lem. 4.7]{BadwaikRufconvergence} where it was derived for partial domains.
\begin{lemma}(Lipschitz continuity in space)\label{lem: Lipschitz continuity in space}
  Let $u$ be the entropy solution of~\eqref{conservation law}. Then the flux $f(k(\cdot),u(\cdot,t))\colon \R\to \R$ is Lipschitz continuous in the sense that
  \begin{equation*}
    \int_0^T |f(k(x),u(x,t)) - f(k(y),u(y,t))|\diff t\leq C|x-y| \qquad \text{for all }x,y\in \R.
  \end{equation*}
\end{lemma}
\begin{proof}
  We first show
  \begin{equation}
    \int_0^T |f^{(i)}(u(x,t)) - f^{(i)}(u(y,t))|\diff t\leq C|x-y| \qquad \text{for all }x,y\in D_i,\ i=0,\ldots,N.
    \label{eqn: intermediate Lipschitz continuity in space}
  \end{equation}
  To that end, we will show~\eqref{eqn: intermediate Lipschitz continuity in space} on $D_0=(-\infty,\xi_1)$ and note that the techniques can be readily adapted for $D_i$, $i=1,\ldots,N$.
  Since $u$ is bounded, we can assume that $\beta\geq(f^{(0)})'\geq \alpha>0$ for some $\alpha,\beta\in\R$. Thus $f^{(0)}$ is invertible with Lipschitz continuous inverse. Setting $w(t,x)=f^{(0)}(u(-x,t))$ and $h(w)=-(f^{(0)})^{-1}$ we find that $w$ satisfies
  \begin{equation*}
    w_x + h(w)_t =0,\qquad (t,x)\in (0,T)\times \R^+.
  \end{equation*}
  Standard theory for conservation laws (with the roles of $x$ and $t$ reversed) adapted to the bounded domain $[0,T]$ shows that $w$ is Lipschitz continuous in $x$ with values in $\mathrm{L}^1(0,T)$, i.e.,
  \begin{equation*}
    \int_0^T |f^{(0)}(u(x,t)) - f^{(0)}(u(y,t))|\diff t = \int_0^T |w(t,-x)-w(t,-y)|\diff t \leq C|x-y|,
  \end{equation*}
  cf. \cite[Thm. 2.15]{holden2015front} or \cite[Lem. 4]{Ridder2019}. Note that an application of \cite[Lem. 4]{Ridder2019} requires, in particular, that
  \begin{align*}
    \TV_{\R^+}(w(0,\cdot)) &= \TV_{\R^-}(f^{(0)}(u_0)),\\
    \TV_{\R^+}(w(T,\cdot)) &= \TV_{\R^-}(f^{(0)}(u(\cdot,T)))\\
    \text{and}\qquad \TV_{[0,T]}(w(\cdot,0+)) &=\TV_{[0,T]}(f^{(0)}(u(0-,t)))
  \end{align*}
  are bounded. The first two quantities are bounded by $C\|f^{(0)}\|_{\mathrm{Lip}}(\TV(k)+\TV(u_0))$ because of~\eqref{eqn: TV bound of adapted entropy solutions} and the third is bounded by $C\|f^{(0)}\|_{\mathrm{Lip}}\TV(u_0)$ due to~\eqref{eqn: temporal TV bound of adapted entropy solutions}.

  Let now $x,y\in\R$. Without restriction we can assume that $y\in D_i$ and $x\in D_{i+j}$. Thus, using the Rankine--Hugoniot condition across the interfaces $\xi_k$ and~\eqref{eqn: intermediate Lipschitz continuity in space} iteratively, we get
  \begin{align*}
    \int_0^T |f(k(x),u(x,t)) & - f(k(y),u(y,t))|\diff t \\
    &\leq \int_0^T \left|f^{(i+j)}(u(x,t)) - f^{(i+j)}(u(\xi_{i+j}+,t))\right|\diff t \\
    &\phantom{\mathrel{=}}+ \sum_{k=1}^{j-1} \int_0^T \left|f^{(i+k)}(u(\xi_{i+k+1}-,t)) - f^{(i+k)}(u(\xi_{i+k}+,t))\right|\diff t\\
    &\phantom{\mathrel{=}}+ \int_0^T \left|f^{(i)}(u(\xi_{i+1}-,t)) - f^{(i)}(u(y,t))\right|\diff t\\
    &\leq C\left(|x-\xi_{i+j}| + \sum_{k=1}^{j-1} |\xi_{i+k+1} - \xi_{i+k}| + |\xi_{i+1} - y| \right)\\
    &= C|x-y|.
  \end{align*}
\end{proof}

\section{Stability for fluxes with one discontinuity}\label{sec: Stability with one discontinuity}
We will first consider the case where $k$ has just two constant values separated by a discontinuity $\xi_1$ and for ease of notation we will assume that $\xi_1=0$. Further, we will denote the flux left of $\xi_1$ as $g$ and right of $\xi_1$ as $f$. In order to prove flux-stability we will derive stability results on $D_0=\R^-$, on $D_1=\R^+$, and on $(0,L)$ for $L>0$ since those three cases represent the fundamental scenarios in the general case.

\subsection{Stability estimates on $\R^-$}\label{subsec: stability on R-}
As a first step we consider the initial value problem
\begin{gather}
\begin{aligned}
  u_t + g(u)_x = 0,& &&(x,t)\in \R^-\times(0,T),\\
  u(x,0) = u_0(x),& &&x\in \R^-
\end{aligned}
\label{conservation law on R-}
\end{gather}
on $\R^-$ with the flux $g$ being strictly monotone and consider entropy solutions in the following sense.
\begin{definition}[Entropy solution on $\R^-$]\label{def: entropy solution on R-}
  We say $u\in\mathcal{C}([0,T];\mathrm{L}^1(\R^-))\cap\mathrm{L}^\infty((0,T)\times\R^-)$ is an entropy solution of~\eqref{conservation law on R-} if for all $c\in\R$,
  \begin{multline*}
    \int_0^T \int_{\R^-} (|u-c|\varphi_t + |g(u)-g(c)|\varphi_x) \diff x\diff t -\int_{\R^-}|u(x,T)-c|\varphi(x,T)\diff x \\
    +\int_{\R^-}|u_0(x)-c|\varphi(x,0)\diff x - \int_0^T |g(u(0-,t))-g(c)|\varphi(0,t)\diff t \geq 0
  \end{multline*}
  for all nonnegative $\varphi\in\mathcal{C}^\infty_c((-\infty,0]\times[0,T])$.
\end{definition}
Note that here $u(0-,t)$ denotes the limit of $u(x,t)$ as $x\to0$ from the left.

In order to prove flux-stability we will use a Kuznetsov-type lemma.
For any function $u\in\mathcal{C}([0,T];\rm{L}^1(\R^-))$ we define
\begin{align*}
  L^-(u,c,\varphi) =& \int_0^T\int_{\R^-} \left( |u-c| \varphi_t +  q(u,c)\varphi_x \right)\diff x\diff t -\int_{\R^-} |u(x,T)-c|\varphi(x,T)\diff x\\
  &+ \int_{\R^-} |u_0(x)-c|\varphi(x,0)\diff x - \int_0^T q(u(0-,t),c)\varphi(0,t)\diff t
\end{align*}
where $q(u,c)=\sign(u-c)(g(u)-g(c))=|g(u)-g(c)|$ is the Kru\v{z}kov entropy flux. Note that if $u$ is an entropy solution of~\eqref{conservation law on R-} then $L^-(u,c,\varphi)\geq 0$ for all $c\in\R$ and test functions $\varphi\geq 0$. We now take $c=v(y,s)$ and the test function
\begin{equation*}
  \varphi(x,t,y,s) = \omega_{\varepsilon}(x-y)\omega_{\varepsilon_0}(t-s)
\end{equation*}
where $\omega_{\varepsilon},\omega_{\varepsilon_0}$ are symmetric standard mollifiers for $\varepsilon,\varepsilon_0>0$. Note that $\varphi_t = -\varphi_s$, $\varphi_x=-\varphi_y$ and
\begin{equation*}
  \varphi(x,t,y,s) = \varphi(y,t,x,s) = \varphi(y,s,x,t) = \varphi(x,s,y,t)
\end{equation*}
as well as
\begin{align*}
\begin{aligned}
  \int_{\R}\omega_{\varepsilon}(x-y)\diff y &\leq 1,\\
  \int_0^T \omega_{\varepsilon_0}(t-s)\diff s  &\leq 1,
\end{aligned}
&&
\begin{aligned}
  \int_{\R} |\omega_{\varepsilon}'(x-y)|\diff y &\leq \frac{C}{\varepsilon},\\
  \int_0^T |\omega_{\varepsilon_0}'(t-s)|\diff s &\leq \frac{C}{\varepsilon_0}
\end{aligned}
\end{align*}
for all $x\in\R$, $t\in [0,T]$.
Let now
\begin{equation*}
  \Lambda_{\varepsilon,\varepsilon_0}^-(u,v) = \int_0^T\int_{\R^-} L^-(u,v(y,s),\varphi(\cdot,\cdot,y,s))\diff y\diff s.
\end{equation*}
For functions $w\in\mathcal{C}([0,T];\mathrm{L}^1(\R^-))$, we further define the moduli of continuity
\begin{align*}
  \nu_t(w,\varepsilon_0) &= \sup_{|\sigma|\leq \varepsilon_0} \|w(\cdot,t+\sigma)-w(\cdot,t)\|_{\mathrm{L}^1(\R^-)},\\
  \mu(w(\cdot,t),\varepsilon) &= \sup_{|z|\leq \varepsilon} \|w(\cdot+z,t)-w(\cdot,t)\|_{\mathrm{L}^1(\R^-)}.
\end{align*}
We have the following Kuznetsov-type lemma which was proved in~\cite{BadwaikRufconvergence}.
\begin{lemma}[{Kuznetsov-type lemma \cite[Lem. 4.2]{BadwaikRufconvergence}}]\label{Lemma: Kuznetsov}
  Let $u$ be the entropy solution of~\eqref{conservation law on R-}. Then, for any function $v:[0,T]\to(\Lone\cap\BV)(\R^-)$ such that the one-sided limits $v(t\pm)$ exist in $\Lone(\R^-)$, we have
  \begin{multline}
    \|u(\cdot,T) - v(\cdot,T)\|_{\mathrm{L}^1(\R^-)} \\+ \int_0^T\int_{\R^-}\int_0^T\left( q(u(0-,t),v(y,s)) + q(v(0-,t),u(y,s)) \right)\varphi(0,t,y,s)\diff t\diff y\diff s\\
    \shoveleft{\leq \|u_0 - v(\cdot,0)\|_{\mathrm{L}^1(\R^-)} - \Lambda_{\varepsilon,\varepsilon_0}^-(v,u)}\\
    + C\bigg( \varepsilon +\varepsilon_0 + \nu_T(v,\varepsilon_0) + \nu_0(v,\varepsilon_0) + \mu(v(\cdot,T),\varepsilon)) + \mu(v(\cdot,0),\varepsilon)) \bigg)
    \label{Kuznetsov lemma estimate on R-}
  \end{multline}
  for some constant $C$ independent of $\varepsilon$ and $\varepsilon_0$.
\end{lemma}
We are now ready to prove flux-stability on $\R^-$ by estimating the term $\Lambda^-_{\varepsilon,\varepsilon_0}(v,u)$ in \Cref{Lemma: Kuznetsov} when $v$ is the entropy solution of~\eqref{conservation law on R-} with a different flux $\tilde{g}$. To that end, we employ similar techniques as in~\cite{lucier1986moving}.
\begin{theorem}[Stability on $\R^-$]\label{thm: stability on R^-}
  Let $u$ be the entropy solution of~\eqref{def: entropy solution on R-} and $v$ the entropy solution of
  \begin{gather}
    \begin{aligned}
      v_t + \tilde{g}(v)_x = 0,& &&(x,t)\in\R^-\times(0,T),\\
      v(x,0)=v_0(x),& &&x\in\R^-
    \end{aligned}
    \label{conservation law on R- for g tilde}
  \end{gather}
  then
  \begin{equation*}
    \|u(\cdot,T) - v(\cdot,T)\|_{\mathrm{L}^1(\R^-)} \leq \|u_0 - v_0\|_{\mathrm{L}^1(\R^-)} + \|g-\tilde{g}\|_{\mathrm{Lip}} \cdot T\cdot \TV(v_0).
  \end{equation*}
\end{theorem}

\begin{proof}
  We use the Kuznetsov-type lemma \ref{Lemma: Kuznetsov} to compare  $v$ to $u$.
  Due to the Lipschitz continuity in time and the TVD property (see \cite[Thm. 2.15 and Lem. A.1]{holden2015front}) the entropy solution $v$ of~\eqref{conservation law on R- for g tilde} satisfies
  \begin{align*}
    \nu_0(v,\varepsilon_0),\,\nu_T(v,\varepsilon_0) &\leq \|\tilde{g}\|_{\mathrm{Lip}}\TV(v_0)\varepsilon_0,\\
    \mu(v(\cdot,0),\varepsilon),~\mu(v(\cdot,T),\varepsilon) &\leq \TV(v_0) \varepsilon,
  \end{align*}
  and the entropy condition
  \begin{multline*}
    \int_0^T\int_{\R^-} (|v-c|\varphi_t + |\tilde{g}(v) - \tilde{g}(c)|\varphi_x )\diff x\diff t
     - \int_{\R^-}|v(x,T) - c|\varphi(x,T)\diff x \\
     + \int_{\R^-}|v_0(x)-c|\varphi(x,0)\diff x - \int_0^T|\tilde{g}(v(0-,t))-\tilde{g}(c)|\varphi(0,t)\diff t \geq 0.
  \end{multline*}
  Thus we can estimate
  \begin{align*}
    -&\Lambda_{\varepsilon,\varepsilon_0}^-(v,u)\\
    &= -\int_0^T\int_{\R^-}\int_0^T\int_{\R^-} \left(|v(x,t)-u(y,s)|\varphi_t + |g(v(x,t))-g(u(y,s))| \varphi_x \right)\diff x\diff t\diff y\diff s\\
     &\mathrel{\phantom{=}}- \int_0^T\int_{\R^-}\int_{\R^-}|v(x,T) - u(y,s)|\varphi(x,T,y,s)\diff x\diff y\diff s\\
     &\mathrel{\phantom{=}}+ \int_0^T\int_{\R^-}\int_{\R^-}|v_0(x)-u(y,s)|\varphi(x,0,y,s)\diff x\diff y\diff s \\
     &\mathrel{\phantom{=}}- \int_0^T|g(v(0-,t))-g(u(y,s))|\varphi(0,t)\diff t\diff y\diff s\\
     &\leq \int_0^T\int_{\R^-}\int_0^T\int_{\R^-}\left( |\tilde{g}(v(x,t))-\tilde{g}(u(y,s))| - |g(v(x,t))-g(u(y,s))|\right)\varphi_x\diff x\diff t\diff y \diff s\\
     &\mathrel{\phantom{=}} \int_0^T\int_{\R^-}\int_0^T\left( |\tilde{g}(v(0-,t))-\tilde{g}(u(y,s))| - |g(v(0-,t))-g(u(y,s))|\right)\varphi(0,t,y,s)\diff t\diff y \diff s\\
     &=\int_0^T\int_{\R^-}\int_0^T\int_{\R^-} \partial_x\left(|g(v(x,t))-g(u(y,s))| - |\tilde{g}(v(x,t))-\tilde{g}(u(y,s))|\right)\varphi\diff x\diff t\diff y\diff s.
  \end{align*}
  We define $h = g-\tilde{g}$. Then
  \begin{equation*}
    |g(v(x,t)) - g(u(y,s))| - |\tilde{g}(v(x,t)) - \tilde{g}(u(y,s))| = \sign(v(x,t) - u(y,s))(h(v(x,t)) - h(u(y,s)))
  \end{equation*}
  which we will define as $H(v(x,t))$. Since $|h(a) - h(b)| = h(a\vee b) - h(a\wedge b)$ for all $a,b\in\R$ we have
  \begin{align*}
    |H(&v(x,t)) - H(v(\tilde{x},t))|\\
    &= | h(v(x,t)\vee u(y,s)) - h(v(x,t)\wedge u(y,s)) - (h(v(\tilde{x},t)\vee u(y,s)) - h(v(\tilde{x},t)\wedge u(y,s)))|\\
    &\leq |h(v(x,t)\vee u(y,s)) - h(v(\tilde{x},t)\vee u(y,s))| + |h(v(x,t)\wedge u(y,s)) - h(v(\tilde{x},t)\wedge u(y,s))|.
  \end{align*}
  By exhausting all the cases, we find that the last line is bounded by
  \begin{equation*}
    \|h\|_{\mathrm{Lip}}\cdot|v(x,t)-v(\tilde{x},t)|.
  \end{equation*}
  Thus we have
  \begin{equation*}
    \frac{|H(v(x,t)) - H(v(\tilde{x},t))|}{|v(x,t)-v(\tilde{x},t)|} \leq \|h\|_{\mathrm{Lip}} = \|g - \tilde{g}\|_{\mathrm{Lip}}
  \end{equation*}
  and hence
  \begin{equation*}
    |\partial_x H(v(x,t))|
  \end{equation*}
  is bounded as a measure by
  \begin{equation*}
    |\partial_v H(v(x,t))\partial_x v(x,t)| \leq \|g-\tilde{g}\|_{\mathrm{Lip}}|\partial_x v(x,t)|.
  \end{equation*}
  This estimate allows us to bound $-\Lambda_{\varepsilon,\varepsilon}(v,u)$ as follows:
  \begin{align*}
    -\Lambda_{\varepsilon,\varepsilon_0}^-(v,u) &\leq \int_0^T\int_{\R^-}\int_0^T\int_{\R^-} \partial_x\left(|g(v(x,t))-g(u(y,s))| - |\tilde{g}(v(x,t))-\tilde{g}(u(y,s))|\right)\varphi\diff x\diff t\diff y\diff s\\
    &\leq \|g-\tilde{g}\|_{\mathrm{Lip}} \int_0^T\int_{\R^-}\int_0^T\int_{\R^-}|\partial_x v(x,t)|\omega_\varepsilon(x-y)\omega_{\varepsilon_0}(t-s)\diff x\diff t\diff y\diff s\\
    &= \|g-\tilde{g}\|_{\mathrm{Lip}} \int_0^T\int_{\R^-} |\partial_x v(y,s)|\diff x\diff t\\
    &\leq \|g-\tilde{g}\|_{\mathrm{Lip}} \cdot T\cdot \TV(v_0).
  \end{align*}
  By inserting these estimates into~\eqref{Kuznetsov lemma estimate on R-}, we obtain
  \begin{multline}
    \|u(\cdot,T) - v(\cdot,T)\|_{\mathrm{L}^1(\R^-)} \\+ \int_0^T\int_{\R^-}\int_0^T\left( q(u(0-,t),v(y,s)) + q(v(0-,t),u(y,s)) \right)\varphi(0,t,y,s)\diff t\diff y\diff s\\
    \leq \|u_0 - v_0\|_{\mathrm{L}^1(\R^-)} + C(\varepsilon + \varepsilon_0) + \|g-\tilde{g}\|_{\mathrm{Lip}} \cdot T\cdot \TV(v_0).
    \label{eqn: almost stability on R^-}
  \end{multline}
  Since $q(u,v) = \sign(u-v)(g(u)-g(v))= |u-v|\geq 0$ the integral term on the left-hand side of \eqref{eqn: almost stability on R^-} is nonnegative such that the left-hand side can be bounded from below by $\|u(\cdot,T)-v(\cdot,T)\|_{\mathrm{L}^1(\R^-)}$.
  The claim follows by passing to the limit $\varepsilon,\varepsilon_0\to 0$.
\end{proof}

\subsection{Stability estimates on $\R^+$}\label{subsec: stability on R+}
As a second step we now consider the initial-boundary value problem
\begin{gather}
\begin{aligned}
  u_t + f(u)_x = 0,& &&(x,t)\in \R^+\times(0,T),\\
  u(x,0) = u_0(x),& &&x\in \R^+,\\
  u(0,t) = f^{-1}\left(g(u(0-,t)\right),& &&t\in (0,T),
\end{aligned}
\label{conservation law on R+}
\end{gather}
where the boundary datum is given in terms of $u(0-,t)$ from the previous section and we consider entropy solutions in the following sense.
\begin{definition}[Entropy solution on $\R^+$]\label{def: entropy solution on R+}
  We say $u\in\mathcal{C}([0,T];\mathrm{L}^1(\R^+))\cap\mathrm{L}^\infty(\R^+\times (0,T))$ is an entropy solution of~\eqref{conservation law on R+} if for all $c\in\R$,
  \begin{multline*}
    \int_0^T \int_{\R^+} (|u-c|\varphi_t + |f(u)-f(c)|\varphi_x) \diff x\diff t -\int_{\R^+}|u(x,T)-c|\varphi(x,T)\diff x\\
    +\int_{\R^+}|u_0(x)-c|\varphi(x,0)\diff x + \int_0^T |f(u(0+,t))-f(c)|\varphi(0,t)\diff t \geq 0
  \end{multline*}
  for all nonnegative $\varphi\in\mathcal{C}^\infty_c([0,\infty)\times[0,T])$ and
  \begin{equation*}
    f(u(0+,t)) = g(u(0-,t))
  \end{equation*}
  holds for almost every $t\in(0,T)$.
\end{definition}

Similarly to before, we define
\begin{align*}
  L^+(u,c,\varphi) =& \int_0^T\int_{\R^+} \left( |u-c| \varphi_t + q(u,c)\varphi_x \right)\diff x\diff t -\int_{\R^+} |u(x,T)-c|\varphi(x,T)\diff x\\
  &+ \int_{\R^+} |u_0(x)-c|\varphi(x,0)\diff x + \int_0^T q(u(0+,t),c)\varphi(0,t)\diff t
\end{align*}
and
\begin{equation}
  \Lambda_{\varepsilon,\varepsilon_0}^+(u,v) = \int_0^T\int_{\R^+} L^+(u,v(y,s),\varphi(\cdot,\cdot,y,s))\diff y\diff s
  \label{eqn: definition Lambda^+}
\end{equation}
where again $\varphi=\omega_{\varepsilon}(x-y)\omega_{\varepsilon_0}(t-s)$.
With this notation, we have the following Kuznetsov-type lemma \cite[Lem. 4.8]{BadwaikRufconvergence}.
\begin{lemma}[Kuznetsov-type lemma \cite{BadwaikRufconvergence}]\label{Lemma: Kuznetsov 2}
  Let $u$ be the entropy solution of~\eqref{conservation law on R+}. Then, for any function $v:[0,T]\to (\Lone\cap\BV)(\R^+)$ such that the one-sided limits $v(t\pm)$ exist in $\Lone$, we have
  \begin{multline*}
    \|u(\cdot,T) - v(\cdot,T)\|_{\mathrm{L}^1(\R^+)} \leq \|u_0 - v(\cdot,0)\|_{\mathrm{L}^1(\R^+)} - \Lambda_{\varepsilon,\varepsilon_0}^+(v,u)\\
    + C\bigg( \varepsilon +\varepsilon_0 + \nu_T(v,\varepsilon_0) + \nu_0(v,\varepsilon_0) + \mu(v(\cdot,T),\varepsilon)) + \mu(v(\cdot,0),\varepsilon)) \bigg) \\
    + \int_0^T\int_{\R^+}\int_0^T\left( q(u(0+,t),v(y,s)) + q(v(0+,t),u(y,s)) \right)\varphi(0,t,y,s)\diff t\diff y\diff s
  \end{multline*}
  for some constant $C$ independent of $\varepsilon$ and $\varepsilon_0$.
\end{lemma}
Note that this time the term involving $q$ is on the right-hand side of the inequality. In order to estimate this spurious term, we will rely on the estimate of the corresponding term on the left-hand side of~\eqref{eqn: almost stability on R^-} as well as on the spatial Lipschitz continuity of the flux from \Cref{lem: Lipschitz continuity in space}.

\begin{theorem}[Stability on $\R^+$]\label{thm: stability on R^+}
  Let $u$ be the entropy solution of~\eqref{def: entropy solution on R+} and $v$ the entropy solution of
  \begin{gather}
    \begin{aligned}
      v_t + \tilde{f}(v)_x = 0,& &&(x,t)\in\R^+\times(0,T),\\
      v(x,0)=v_0(x),& &&x\in\R^+\\
      v(0,t) = \tilde{f}^{-1}(\tilde{g}(v(0-,t))),& &&t\in(0,T),
    \end{aligned}
    \label{conservation law on R+ for f tilde}
  \end{gather}
  where the boundary datum is given in terms of $v(0-,t)$ from \Cref{thm: stability on R^-}.
  Then
  \begin{multline}
    \|u(\cdot,T) - v(\cdot,T)\|_{\mathrm{L}^1(\R^+)} \\
    \leq \|u_0 - v_0\|_{\mathrm{L}^1(\R)} + CT \left(\|g-\tilde{g}\|_{\mathrm{Lip}} + \|f-\tilde{f}\|_{\mathrm{Lip}} \right) \left(\TV(v_0) + \|v_0\|_{\mathrm{L}^\infty(\R)} \right).
    \label{eqn: stability on R^+}
  \end{multline}
\end{theorem}
\begin{proof}
  We use the Kuznetsov-type lemma \ref{Lemma: Kuznetsov 2}. Due to the Lipschitz continuity in time and the TVD property (for the TVD property of conservation laws on bounded domains see \cite[Lem. 2]{Ridder2019}) the entropy solution $v$ of~\eqref{conservation law on R+ for f tilde} satisfies
  \begin{align*}
    \nu_0(v,\varepsilon_0),\,\nu_T(v,\varepsilon_0) &\leq \|\tilde{f}\|_{\mathrm{Lip}}\TV(v_0)\varepsilon_0,\\
    \mu(v(\cdot,0),\varepsilon),~\mu(v(\cdot,T),\varepsilon) &\leq \TV(v_0) \varepsilon.
  \end{align*}
  In order to estimate $-\Lambda_{\varepsilon,\varepsilon_0}^+(v,u)$, we can follow the same steps as in the proof of \Cref{thm: stability on R^-} to obtain
  \begin{equation*}
    -\Lambda_{\varepsilon,\varepsilon_0}^+(v,u) \leq \|f-\tilde{f}\|_{\mathrm{Lip}}\cdot T\cdot\TV(v_0).
  \end{equation*}
  It remains to estimate the term
  \begin{align*}
    \int_0^T & \int_{\R^+}\int_0^T \left( q(u(0+,t),v(y,s)) + q(v(0+,t),u(y,s)) \right)\varphi(0,t,y,s)\diff t\diff y\diff s \\
    &=\int_0^T\int_{\R^+}\int_0^T\left( \underbrace{|f(u(0+,t))- f(v(y,s))|}_{\eqqcolon \mathbf{A}} + \underbrace{|f(v(0+,t) - f(u(y,s))|}_{\eqqcolon \mathbf{B}} \right)\varphi(0,t,y,s)\diff t\diff y\diff s.
  \end{align*}
  \paragraph{Ad $\mathbf{A}$:}
  Here, we split
  \begin{equation*}
  	\mathbf{A} \leq \underbrace{|f(u(0+,t)) - \tilde{f}(v(0+,s))|}_{\eqqcolon\mathbf{A_1}} + \underbrace{|\tilde{f}(v(0+,s)) - \tilde{f}(v(y,s))|}_{\eqqcolon\mathbf{A_2}} + \underbrace{|\tilde{f}(v(y,s)) - f(v(y,s))|}_{\eqqcolon\mathbf{A_3}}
  \end{equation*}
  Regarding $\mathbf{A_2}$ we use the Lipschitz continuity of $\tilde{f}(v)$ in space from \Cref{lem: Lipschitz continuity in space} to get
  \begin{align*}
  	\int_0^T\int_{\R^+}\int_0^T \mathbf{A_2} \varphi(0,t,y,s) \diff t\diff y\diff s &= \int_0^T\int_{\R^+}\int_0^T |\tilde{f}(v(0+,s)) - \tilde{f}(v(y,s))| \varphi(0,t,y,s) \diff t\diff y\diff s\\
  	&\leq \frac{1}{\epsilon} \int_0^\varepsilon \underbrace{\int_0^T  |\tilde{f}(v(0+,s)) - \tilde{f}(v(y,s))| \diff s}_{\leq C|y|} \diff y\\
  	&\leq \frac{C}{\varepsilon} \int_0^\varepsilon |y|\diff y\\
  	&=C\varepsilon.
  \end{align*}
  In order to estimate the term $\mathbf{A_3}$ we use the $\mathrm{L}^\infty$ bound~\eqref{eqn: L infty bound of adapted entropy solutions} and the fact that $f(0)=\tilde{f}(0)$ to get
  \begin{align*}
  	\int_0^T\int_{\R^+}\int_0^T \mathbf{A_3} \varphi(0,t,y,s) \diff t\diff y\diff s &= \int_0^T\int_{\R^+}\int_0^T |\tilde{f}(v(y,s)) - f(v(y,s))| \varphi(0,t,y,s) \diff t\diff y\diff s\\
  	&= \int_0^T\int_{\R^+}\int_0^T |(\tilde{f}-f)(v(y,s)) - (\tilde{f}-f)(0)| \varphi(0,t,y,s) \diff t\diff y\diff s\\
  	&\leq \|f-\tilde{f}\|_{\mathrm{Lip}} \int_0^T\int_{\R^+} |v(y,s)|\omega_\varepsilon(-y)\diff y\diff s\\
  	&\leq \|f-\tilde{f}\|_{\mathrm{Lip}} C T \|v\|_{\mathrm{L}^\infty(\R^+\times (0,T))}\\
  	&\leq \|f-\tilde{f}\|_{\mathrm{Lip}} C T \|v_0\|_{\mathrm{L}^\infty(\R^+)}
  \end{align*}
  For the term $\mathbf{A_1}$ we use the Rankine--Hugoniot condition to cross the interface at $x=0$ such that we can utilize the bound for the integral term on the left-hand side of the stability estimate on $\R^-$ (cf.~\eqref{eqn: almost stability on R^-}). To that end, we estimate
  \begin{align*}
  	\mathbf{A_1} &= |f(u(0+,t)) - \tilde{f}(v(0+,s))|\\
  	&= |g(u(0-,t)) - \tilde{g}(v(0-,s))|\\
  	&\leq |g(u(0-,t)) - g(v(y,s))| + |g(v(y,s)) - \tilde{g}(v(y,s))| + |\tilde{g}(v(y,s)) - \tilde{g}(v(0-,s))|.
  \end{align*}
  Since $\mathbf{A_1}$ is independent of $y$, we can use the symmetry of $\varphi$ with respect to $y$ to obtain
  \begin{align*}
  	\int_0^T & \int_{\R^+}\int_0^T \mathbf{A_1} \varphi(0,t,y,s)\diff t\diff y\diff s\\
  	&=\int_0^T \int_{\R^-}\int_0^T \mathbf{A_1} \varphi(0,t,y,s)\diff t\diff y\diff s\\
  	&\leq \int_0^T \int_{\R^-}\int_0^T |g(u(0-,t)) - g(v(y,s))| \varphi(0,t,y,s)\diff t\diff y\diff s\\
  	&\phantom{\mathrel{=}} \int_0^T \int_{\R^-}\int_0^T |g(v(y,s)) - \tilde{g}(v(y,s))| \varphi(0,t,y,s)\diff t\diff y\diff s\\
  	&\phantom{\mathrel{=}} \int_0^T \int_{\R^-}\int_0^T |\tilde{g}(v(y,s)) - \tilde{g}(v(0-,s))| \varphi(0,t,y,s)\diff t\diff y\diff s.
  \end{align*}
  Note that the first integral in this estimate is bounded by the right-hand side of~\eqref{eqn: almost stability on R^-} such that we have
  \begin{multline*}
  	\int_0^T \int_{\R^-}\int_0^T |g(u(0-,t)) - g(v(y,s))| \varphi(0,t,y,s)\diff t\diff y\diff s\\
  	\leq \|u_0 - v_0\|_{\mathrm{L}^1(\R^-)} + C(\varepsilon + \varepsilon_0) + \|g-\tilde{g}\|_{\mathrm{Lip}}\cdot T\cdot \TV(v_0).
  \end{multline*}
  Using the same techniques as for the term $\mathbf{A_3}$, the second integral can be estimated by
  \begin{equation*}
  	\|g-\tilde{g}\|_{\mathrm{Lip}} C T \|v_0\|_{\mathrm{L}^\infty(\R^-)}.
  \end{equation*}
  The third integral is bounded by $C\varepsilon$ which can be seen by employing the Lipschitz continuity of $\tilde{g}(v)$ in space just as for the term $\mathbf{A_2}$.
  In summary, we have
  \begin{multline*}
  	\int_0^T\int_{\R^+}\int_0^T \mathbf{A}\varphi(0,t,y,s) \diff t\diff y\diff s \\
  	\leq \|u_0 - v_0\|_{\mathrm{L}^1(\R^-)} + C(\varepsilon + \varepsilon_0) + \|g-\tilde{g}\|_{\mathrm{Lip}}\cdot T\cdot \TV(v_0) \\
  	+ ( \|f-\tilde{f}\|_{\mathrm{Lip}} + \|g-\tilde{g}\|_{\mathrm{Lip}}) C T \|v_0\|_{\mathrm{L}^\infty(\R)}
  \end{multline*}
  \paragraph{Ad $\mathbf{B}$:} It remains to estimate the term $\mathbf{B}$ which we split as follows:
  \begin{align*}
  	\mathbf{B} &= |f(v(0+,t)) - f(u(y,s))| \\
  	&\leq |f(v(0+,t)) - \tilde{f}(v(0+,t))| + |\tilde{f}(v(0+,t)) - f(u(0+,s))| + |f(u(0+,s)) - f(u(y,s))|.
  \end{align*}
  From here, we can apply the same steps, \emph{mutatis mutandis}, as for the term $\mathbf{A}$ and we ultimately get
  \begin{multline*}
  	\int_0^T\int_{\R^+}\int_0^T \mathbf{B}\varphi(0,t,y,s) \diff t\diff y\diff s \\
  	\leq \|u_0 - v_0\|_{\mathrm{L}^1(\R^-)} + C(\varepsilon + \varepsilon_0) + \|g-\tilde{g}\|_{\mathrm{Lip}}\cdot T\cdot \TV(v_0) \\
  	+ ( \|f-\tilde{f}\|_{\mathrm{Lip}} + \|g-\tilde{g}\|_{\mathrm{Lip}}) C T \|v_0\|_{\mathrm{L}^\infty(\R)}
  \end{multline*}
  Combining all the bounds and passing to the limit $\varepsilon,\varepsilon_0\to 0$ yields the desired estimate.
\end{proof}

\subsection{Stability estimates on $(0,L)$}

By restricting the solutions $u$ and $v$ to a bounded interval $(0,L)$ Theorems \ref{thm: stability on R^+} and~\ref{thm: stability on R^-} yield a stability estimate on $(0,L)$ as well.

\begin{corollary}[Stability on $(0,L)$]\label{cor: stability on bounded domain}
  Let $u$ be the entropy solution of~\eqref{conservation law on R+} on the bounded interval $(0,L)$ and $v$ the entropy solution of
  \begin{gather*}
    \begin{aligned}
      v_t + \tilde{f}(v)_x = 0,& &&(x,t)\in(0,L)\times(0,T),\\
      v(x,0)=v_0(x),& &&x\in(0,L)\\
      v(0,t) = \tilde{f}^{-1}(\tilde{g}(v(0-,t))),& &&t\in(0,T),
    \end{aligned}
  \end{gather*}
  where the boundary datum is given in terms of $v(0-,t)$ from \Cref{thm: stability on R^-}.
  Then
  \begin{multline}
    \|u(\cdot,T) - v(\cdot,T)\|_{\mathrm{L}^1(0,L)} \\
    \leq \|u_0 - v_0\|_{\mathrm{L}^1(-\infty,L)} + CT \left(\|g-\tilde{g}\|_{\mathrm{Lip}} + \|f-\tilde{f}\|_{\mathrm{Lip}} \right) \left(\TV(v_0) + \|v_0\|_{\mathrm{L}^\infty(\R)} \right).
    \label{eqn: stability on R^+}
  \end{multline}
\end{corollary}
\begin{proof}
  Without repeating all calculations of Sections~\ref{subsec: stability on R-} and~\ref{subsec: stability on R+} we will highlight the adjustments to the respective proofs that need to be done. If we consider solutions on $(0,L)$ instead of $\R^+$ the definition of $\Lambda_{\varepsilon,\varepsilon_0}^{(0,L)}(u,v)$ in~\eqref{eqn: definition Lambda^+} needs to be adjusted so that $\Lambda_{\varepsilon,\varepsilon_0}^{(0,L)}(u,v)$ contains the term
  \begin{equation*}
    -\int_0^T\int_0^L\int_0^T q(u(L-,t),v(y,s))\varphi(L,t,y,s) \diff t\diff y\diff s
  \end{equation*}
  and all instances of $\R^+$ need to be changed to $(0,L)$. 
  Following the relevant techniques from the proofs of Theorems~\ref{thm: stability on R^-} and~\ref{thm: stability on R^+} in the same way ultimately yields
  \begin{multline}
    \|u(\cdot,T) - u_\Dt(\cdot,T)\|_{\mathrm{L}^1(0,L)} \\
  	+ \int_0^T\int_0^L\int_0^T\left( q(u(L,t),u_\Dt(y,s)) + q(v(L-,t),u(y,s)) \right)\varphi(L,t,y,s)\diff t\diff y\diff s\\
    \leq \|u_0 - v_0\|_{\mathrm{L}^1(-\infty,L)} + CT \left(\|g-\tilde{g}\|_{\mathrm{Lip}} + \|f-\tilde{f}\|_{\mathrm{Lip}} \right) \left(\TV(v_0) + \|v_0\|_{\mathrm{L}^\infty(\R)} \right).
    \label{Almost convergence rate on (0,L)}
  \end{multline}
  Using the monotonicity of $f$ we find
  \begin{equation*}
    q(u,v) = |f(u)-f(v)|\geq 0
  \end{equation*}
  and thus the integral term in~\eqref{Almost convergence rate on (0,L)} is nonnegative which concludes the proof.
\end{proof}

\section{Statement and proof of the main theorem}\label{sec: main stability result}
Our main result now reads as follows.


\begin{theorem}[Stability for conservation laws with discontinuous flux]\label{thm: main stability estimate}
  Let $f,\tilde{f}\in\mathcal{C}^2(\R^2;\R)$ be strictly monotone in $u$, i.e., $f_u\geq \alpha>0$ and $\tilde{f}_u\geq\tilde{\alpha}>0$, let $k$ and $\tilde{k}$ be piecewise constant, each having finitely many discontinuities, and $u_0,v_0\in\left(\mathrm{L}^1\cap\BV\right)(\R)$. Let $u$ and $v$ be entropy solutions of~\eqref{conservation law} in the sense of \Cref{def: entropy solution} with fluxes $f(k(\cdot),\cdot)$ and $\tilde{f}(\tilde{k}(\cdot),\cdot)$ and initial data $u_0$ and $v_0$, respectively. Then
  \begin{equation*}
    \|u(\cdot,T) - v(\cdot,T)\|_{\mathrm{L}^1(\R)} \leq C\left(\|u_0 - v_0\|_{\mathrm{L}^1(\R)} + \|k-\tilde{k}\|_{\mathrm{L}^\infty(\R)} + \max_{x\in\R} \|f(k(x),\cdot) - \tilde{f}(\tilde{k}(x),\cdot)\|_{\mathrm{Lip}}\right)
  \end{equation*}
  where the constant $C$ depends on $u_0,v_0$, the Lipschitz constants of $f$ and $\tilde{f}$ and the number of discontinuities in $k$ and $\tilde{k}$.
\end{theorem}

\begin{proof}
  We proceed in three steps. First we keep the coefficient $k$ the same and show stability with respect to changes in $f$ using the results previously derived. Then we show stability in the coefficient $k$ and lastly we combine both estimates.

  \emph{First step:} Let $u$ and $v$ be entropy solutions of~\eqref{conservation law} with fluxes $f(k(\cdot),\cdot)$ and $\tilde{f}(k(\cdot),\cdot)$ and initial data $u_0$ and $v_0$, respectively. As seen before, we decompose the entropy solution $u$ as $u=\sum_{i=0}^N u^{(i)}$, where $u^{(i)}$ are the respective entropy solutions on $D_i$, i.e., solutions of~\eqref{conservation law on D0} and~\eqref{conservation law on Di}, respectively. We decompose $v$ in the same manner. Then, we have
  \begin{equation*}
    \|u(\cdot,T) - v(\cdot,T)\|_{\mathrm{L}^1(\R)} = \sum_{i=0}^{N} \|u^{(i)}(\cdot,T) - v^{(i)}(\cdot,T)\|_{\mathrm{L}^1(D_i)}.
  \end{equation*}
  By applying \Cref{thm: stability on R^-} for $D_0$, \Cref{cor: stability on bounded domain} iteratively for each $D_i$, $i=1,\ldots,N-1$, and finally \Cref{thm: stability on R^+} for $D_N$ we obtain
  \begin{equation}
    \|u(\cdot,T) - v(\cdot,T)\|_{\mathrm{L}^1(\R)} \leq C\left(\|u_0 - v_0\|_{\mathrm{L}^1(\R)} + \max_{i=0,\ldots,N} \|f^{(i)} - \tilde{f}^{(i)}\|_{\mathrm{Lip}}\right).
    \label{eqn: stability in the flux only}
  \end{equation}

  \emph{Second step:} Let now $k$ and $\tilde{k}$ be piecewise constant functions each with finitely many discontinuities. Denote by $(\zeta_1,\ldots,\zeta_M)$ the smallest ordered set containing the discontinuities of both $k$ and $\tilde{k}$ and let
  \begin{equation*}
    k_i = k(x),\qquad \tilde{k}_i = \tilde{k}(x)\qquad \text{for }x\in (\zeta_i,\zeta_{i+1}),\ i=0,\ldots,M
  \end{equation*}
  where we have used the notation $\zeta_0=-\infty$ and $\zeta_{M+1}=\infty$. Then we define the sequences of fluxes
  \begin{align*}
    \begin{aligned}
    f^{(i)} &\coloneqq f(k_i,\cdot) = f(k(x),\cdot)\\
    \tilde{f}^{(i)} &\coloneqq f(\tilde{k}_i,\cdot) = f(\tilde{k}(x),\cdot)
    \end{aligned}
    \qquad\text{for }x\in (\zeta_i,\zeta_{i+1}),\ i=0,\ldots,M.
  \end{align*}
  Note that for given $i$ the fluxes $f^{(i)}$ and $f^{(i+1)}$ (or $\tilde{f}^{(i)}$ and $\tilde{f}^{(i+1)}$) might be identical since $\zeta_{i+1}$ might only be a discontinuity of one coefficient of $\{k,\tilde{k}\}$ and not both.

  In the first step we showed stability with respect to changes in $f$ by decomposing the spatial domain into finitely many intervals $(D_i)_{i=0}^N$ with endpoints given by the discontinuities of $k$. We will now argue that the same results holds for any decomposition of the spatial domain into finitely many intervals $(I_i)_{i=0}^M$ as long as the discontinuities of $k$ are among the endpoints of the intervals $I_i$. Since an entropy solution has bounded variation its traces at each point where $k$ is continuous exist and the traces satisfy the Rankine--Hugoniot condition which is enough to carry out the same arguments as in \Cref{sec: Stability with one discontinuity}.

  Now let $u$ and $w$ be the entropy solutions of~\eqref{conservation law} with fluxes $f(k(\cdot),\cdot)$ and $f(\tilde{k}(\cdot),\cdot)$, respectively, and with the same initial datum $u_0$. Then, using~\eqref{eqn: stability in the flux only} and the mean value theorem, we obtain
  \begin{align*}
    \|u(\cdot,T) - w(\cdot,T)\|_{\mathrm{L}^1(\R)} &\leq C \max_{i=0,\ldots,M} \|f^{(i)} - \tilde{f}^{(i)}\|_{\mathrm{Lip}}\\
    &\leq C \max_{i=0,\ldots,M} \sup_u |f_u(k_i,u) - f_u(\tilde{k}_i,u)|\\
    &= C \max_{i=0,\ldots,M} \sup_u |f_{uk}(\xi,u)| |k_i-\tilde{k}_i|\\
    &\leq C\max_{i=0,\ldots,M} |k_i - \tilde{k}_i|\\
    &= C\|k-\tilde{k}\|_{\mathrm{L}^\infty(\R)}.
  \end{align*}
  Here, in order to bound the term $\sup_u |f_{uk}(\xi,u)|$, we have used the fact that $f\in\mathcal{C}^2$ and the values of $u$ and $\xi$ are in 
  \begin{equation*}
  [-C\|u_0\|_{\mathrm{L}^\infty(\R)},C\|u_0\|_{\mathrm{L}^\infty(\R)}] \qquad\text{and}\qquad
  [-\max(\|k\|_{\mathrm{L}^\infty(\R)},\|\tilde{k}\|_{\mathrm{L}^\infty(\R)}), \max(\|k\|_{\mathrm{L}^\infty(\R)},\|\tilde{k}\|_{\mathrm{L}^\infty(\R)})].
  \end{equation*}

  \emph{Third step:} Finally, let $u$ and $v$ be the entropy solutions of~\eqref{conservation law} with fluxes $f(k(\cdot),\cdot)$ and $\tilde{f}(\tilde{k}(\cdot),\cdot)$ and initial data $u_0$ and $v_0$, respectively. Let further $w$ be the entropy solution of~\eqref{conservation law} with flux $f(\tilde{k}(\cdot),\cdot)$ and initial datum $u_0$. Then, using the triangle inequality and steps one and two, we obtain
  \begin{gather}
  \begin{aligned}
    \|u(\cdot,T) - v(\cdot,T)\|_{\mathrm{L}^1(\R)} &\leq \|u(\cdot,T) - w(\cdot,T)\|_{\mathrm{L}^1(\R)} + \|w(\cdot,T) - v(\cdot,T)\|_{\mathrm{L}^1(\R)}\\
    &\leq C \left(\|u_0 - v_0\|_{\mathrm{L}^1(\R)} + \|k-\tilde{k}\|_{\mathrm{L}^\infty(\R)} + \max_{i=0,\ldots,N} \|f^{(i)} - \tilde{f}^{(i)}\|_{\mathrm{Lip}} \right).
  \end{aligned}
  \label{eqn: almost stability}
  \end{gather}
  Like in step two, we have
  \begin{align*}
    \max_{i=0,\ldots,N} \|f^{(i)} - \tilde{f}^{(i)}\|_{\mathrm{Lip}} &= \max_{x\in\R} \|f(\tilde{k}(x),\cdot) - \tilde{f}(\tilde{k}(x),\cdot)\|_{\mathrm{Lip}}\\
    &\leq \max_{x\in\R} \left( \|f(\tilde{k}(x),\cdot) - f(k(x),\cdot)\|_{\mathrm{Lip}} + \|f(k(x),\cdot) - \tilde{f}(\tilde{k}(x),\cdot)\|_{\mathrm{Lip}} \right)\\
    &\leq C \|k-\tilde{k}\|_{\mathrm{L}^\infty(\R)} + \max_{x\in\R} \|f(k(x),\cdot) - \tilde{f}(\tilde{k}(x),\cdot)\|_{\mathrm{Lip}}
  \end{align*}
  which together with~\eqref{eqn: almost stability} yields the desired stability result.
\end{proof}

\begin{remark}
	We want to mention that the stability result of \Cref{thm: main stability estimate} can not only be used to prove a convergence rate for the front tracking method as in the following section, but also in the context of uncertainty quantification for~\eqref{conservation law} where the model parameters $u_0$, $k$, and $f$ are subject to randomness \cite{badwaik2019multilevel}. In that case, the stability result is integral in showing existence and uniqueness of random entropy solutions.
	The stability result can also be utilized to show well-posedness of Bayesian inverse problems for conservation laws with discontinuous flux \cite{MishraBayesian}.
\end{remark}

\section{Convergence rate estimates for the front tracking method}\label{sec: Convergence rate front tracking}
The front tracking method consists of approximating the initial datum by a piecewise constant function, the flux by a piecewise linear function, and solving the resulting approximated conservation law exactly. To that end, the algorithm needs to solve a series of Riemann problems across discontinuities of $k$. Therefore, we first introduce the solution of such a Riemann problem before we present the front tracking method and prove its convergence rate in the following subsections.

\subsection{The Riemann solver}
Let us consider the Riemann problem for~\eqref{conservation law} first. We assume that the discontinuities in $k$ and $u_0$ coincide since at the points where $x\mapsto f(k(x),u)$ is continuous, the Riemann problems are the usual Riemann problems for scalar conservation laws. The initial datum and the parameter $k$ are piecewise constant functions defined by
\begin{equation}
  u_0(x) = \begin{cases}
    u_L,& x<0,\\
    u_R,& x>0,
  \end{cases}
  \qquad k(x) = \begin{cases}
    k_L,& x<0,\\
    k_R,& x>0,
  \end{cases}
  \label{eqn: data for Riemann problem}
\end{equation}
such that the flux can be written as
\begin{equation*}
  f(k(x),u) = \begin{cases}
    g(u),& x<0,\\
    f(u),& x>0,
  \end{cases}
\end{equation*}
for some strictly increasing functions $f,g$. The solution of the Riemann problem now is an entropy solution of
\begin{gather}
\begin{aligned}
  u_t + g(u)_x = 0,& &&(x,t)\in\R^-\times(0,T),\\
  u_t + f(u)_x = 0,& &&(x,t)\in\R^+\times(0,T),\\
  u(x,0)=\begin{cases}
  	u_L,& x<0,\\
  	u_R,& x>0,
  \end{cases}& && x\in\R.
\end{aligned}
\label{eqn: Riemann problem}
\end{gather}
\begin{proposition}
	Consider the Riemann problem~\eqref{eqn: Riemann problem} and let $u^* \coloneqq f^{-1}(g(u_L))$.
	Then the unique entropy solution of the Riemann problem~\eqref{eqn: Riemann problem} is
	\begin{equation*}
	u(x,t) = \begin{cases}
			u_L,& x<0,\\
			\left(R_f(u^*,u_R)\right)(x,t),& x>0,
		\end{cases}
	\end{equation*}
	where $R_f(u^*,u_R)$ denotes the standard solution of the Riemann problem
	\begin{gather}
	\begin{aligned}
	  u_t + f(u)_x = 0,& &&(x,t)\in\R\times(0,T),\\
	  u(x,0)=\begin{cases}
	  	u^*,& x<0,\\
	  	u_R,& x>0,
	  \end{cases}& &&x\in\R.
	\end{aligned}
	\label{eqn: standard Riemann problem}
	\end{gather}
\end{proposition}
\begin{proof}
	The proof has been carried out in~\cite[Prop. 1]{Piccoli/Tournus} for the case of $f$ and $g$ (strictly increasing and) concave and can be easily modified to the general case in view of~\cite[Thm 2.2]{holden2015front}.
\end{proof}
The solution of the classical Riemann problem $R_f(u^*,u_R)$ consists of a finite sequence of rarefaction waves alternating with shocks determined by the number of inflection points of $f$. If $f$ is convex, the solution of the classical Riemann problem is given by a single shock if $u^*>u_R$ and by a single rarefaction wave if $u^*<u_R$ (see \Cref{fig: solution of the Riemann problem}).
%
%
%
%
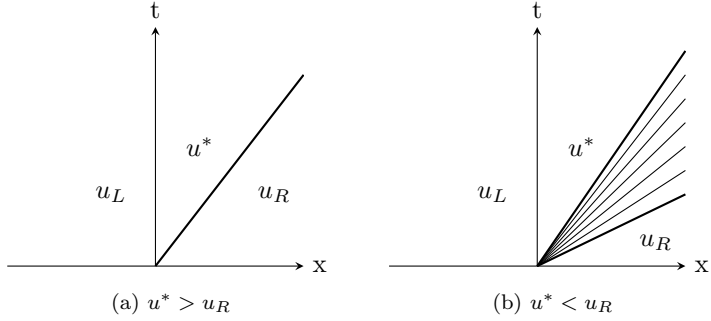
\begin{figure}[t]
\centering
\subfloat[$u^*>u_R$]{
\begin{tikzpicture}
  \begin{axis}[xmin=-1,xmax=1,ymax=1,axis lines=center,xtick={0},xlabel={x},ytick=\empty,ylabel={t},every axis x label/.style={
    at={(ticklabel* cs:1.)},
    anchor=west,
},every axis y label/.style={
    at={(ticklabel* cs:1.)},
    anchor=south,
}]
    \addplot[sharp plot,mark=none, thick,black] coordinates{
    (0,0)
    (1,0.8)
    };
    \node at (0.8,0.3) {$u_R$};
    \node at (-0.3,0.3) {$u_L$};
    \node at (0.3,0.5) {$u^*$};
  \end{axis}
\end{tikzpicture}
}
\hspace{1em}
\subfloat[$u^*<u_R$]{
\begin{tikzpicture}
  \begin{axis}[xmin=-1,xmax=1,ymax=1,axis lines=center,xtick={0},xlabel={x},ytick=\empty,ylabel={t},every axis x label/.style={
    at={(ticklabel* cs:1.)},
    anchor=west,
},every axis y label/.style={
    at={(ticklabel* cs:1.)},
    anchor=south,
}]
    \addplot[sharp plot,mark=none, thick,black] coordinates{
    (0,0)
    (1,0.3)
    };
    \addplot[sharp plot,mark=none,black] coordinates{
    (0,0)
    (1,0.4)
    };
    \addplot[sharp plot,mark=none,black] coordinates{
    (0,0)
    (1,0.5)
    };
    \addplot[sharp plot,mark=none,black] coordinates{
    (0,0)
    (1,0.6)
    };
    \addplot[sharp plot,mark=none,black] coordinates{
    (0,0)
    (1,0.7)
    };
    \addplot[sharp plot,mark=none,black] coordinates{
    (0,0)
    (1,0.8)
    };
    \addplot[sharp plot,mark=none, thick,black] coordinates{
    (0,0)
    (1,0.9)
    };
    \node at (0.8,0.1) {$u_R$};
    \node at (-0.3,0.3) {$u_L$};
    \node at (0.3,0.5) {$u^*$};
  \end{axis}
\end{tikzpicture}
}
\caption{The entropy solution of the Riemann problem~\eqref{eqn: Riemann problem} for $f$ convex consists of a stationary shock wave at the spatial discontinuity and a shock or a rarefaction wave.}
\label{fig: solution of the Riemann problem}
\end{figure}%
\begin{remark}\label{rem: Form of the Riemann problem solution for pcw. linear flux}
	If in~\eqref{eqn: Riemann problem} the flux $f$ is approximated by a piecewise linear flux $f_\delta$ then the solution of the Riemann problem~\eqref{eqn: standard Riemann problem} (with $f_\delta$ as the flux) $R_{f_\delta}(u^*,u_R)$ is an approximation to $R_f(u^*,u_R)$ in the sense that any shock in $R_f(u^*,u_R)$ is also present in $R_{f_\delta}(u^*,u_R)$ and any rarefaction wave in $R_f(u^*,u_R)$ is approximated by a so-called rarefaction fan consisting of a series of rarefaction shocks whose strength is bounded by the approximation error between $f$ and $f_\delta$ (cf. \cite[Lemma 3.1]{dafermos1972polygonal} or \cite[Cor. 2.4]{holden2015front}). Moreover, the piecewise constant solution takes values in the set of breakpoints of $f_\delta$ and $\{u_L,u^*,u_R\}$. See \Cref{fig: solution of the Riemann problem with piecewise linear flux} for an illustration in the case of a convex flux $f$.
\end{remark}%
\begin{figure}
\centering
\subfloat[$u^*>u_R$]{
\begin{tikzpicture}
  \begin{axis}[xmin=-1,xmax=1,ymax=1,axis lines=center,xtick={0},xlabel={x},ytick=\empty,ylabel={t},every axis x label/.style={
    at={(ticklabel* cs:1.)},
    anchor=west,
},every axis y label/.style={
    at={(ticklabel* cs:1.)},
    anchor=south,
}]
    \addplot[sharp plot,mark=none, thick,black] coordinates{
    (0,0)
    (1,0.8)
    };
    \node at (0.8,0.3) {$u_R$};
    \node at (-0.3,0.3) {$u_L$};
    \node at (0.3,0.5) {$u^*$};
  \end{axis}
\end{tikzpicture}
}
\hspace{1em}
\subfloat[$u^*<u_R$]{
\begin{tikzpicture}
  \begin{axis}[xmin=-1,xmax=1,ymax=1,axis lines=center,xtick={0},xlabel={x},ytick=\empty,ylabel={t},every axis x label/.style={
    at={(ticklabel* cs:1.)},
    anchor=west,
},every axis y label/.style={
    at={(ticklabel* cs:1.)},
    anchor=south,
}]
    \addplot[sharp plot,mark=none, thick,black] coordinates{
    (0,0)
    (1,0.3)
    };
    \addplot[sharp plot,mark=none,thick,black] coordinates{
    (0,0)
    (1,0.4)
    };
    \addplot[sharp plot,mark=none,thick,black] coordinates{
    (0,0)
    (1,0.6)
    };
    \addplot[sharp plot,mark=none,thick,black] coordinates{
    (0,0)
    (1,0.8)
    };
    \addplot[sharp plot,mark=none, thick,black] coordinates{
    (0,0)
    (1,0.9)
    };
    \node at (0.8,0.1) {$u_R$};
    \node at (-0.3,0.3) {$u_L$};
    \node at (0.3,0.5) {$u^*$};
  \end{axis}
\end{tikzpicture}
}
\caption{The entropy solution of the Riemann problem~\eqref{eqn: Riemann problem} for $f_\delta$ piecewise linear and convex consists of a stationary shock wave at the spatial discontinuity and a shock or a series of rarefaction shocks each of strength at most $\delta$.}
\label{fig: solution of the Riemann problem with piecewise linear flux}
\end{figure}
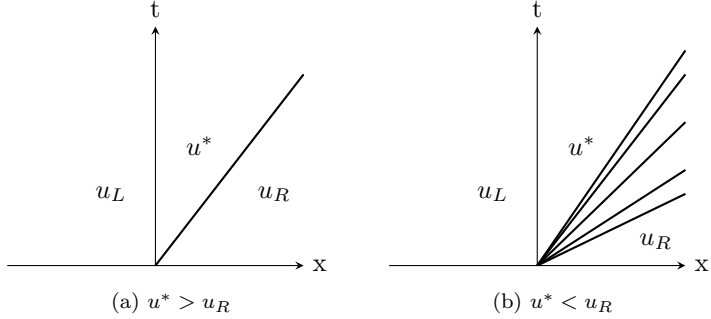%
\subsection{The front tracking method}
Let us now outline the front tracking method for~\eqref{conservation law} (see \cite{Piccoli/Tournus} or \cite{dafermos1972polygonal,holden2015front}). Given a piecewise constant function $k$ with finitely many discontinuities $\xi_i$, $i=1,\ldots,N$, we define -- as before -- $D_i=(\xi_i,\xi_{i+1})$, $i=0,\ldots,N$, where $\xi_0=-\infty$ and $\xi_{N+1}=\infty$. Further, we denote
\begin{equation*}
	f^{(i)}(\cdot) \coloneqq f(k(x),\cdot)\qquad \text{for }x\in D_i.
\end{equation*}

For the front tracking method, we approximate the initial datum $u_0$ by a piecewise constant function $u_0^\delta$ with finitely many discontinuities and we approximate the $u$-dependent fluxes $f^{(i)}$ by piecewise linear functions $f_\delta^{(i)}$ and define $f_\delta(k(x),\cdot)\coloneqq f_\delta^{(i)}(\cdot)$ for $x\in D_i$. Then we solve the surrogate conservation law
\begin{gather}
	\begin{aligned}
		u_t + f_\delta(k(x),u)_x =0,& &&(x,t)\in\R\times(0,T)\\
		u(x,0)=u_0^\delta(x),& &&x\in\R
	\end{aligned}
	\label{eqn: approximate conservation law}
\end{gather}
exactly. Since $u_0^\delta$ is piecewise constant, we have to solve a series of independent Riemann problems of the kind~\eqref{eqn: Riemann problem}. Since each flux $f_\delta^{(i)}$ is piecewise linear the solutions of those Riemann problems will consists of a series of shocks each traveling with constant speed (see \Cref{rem: Form of the Riemann problem solution for pcw. linear flux}). When two shock fronts meet at time $t=t^*$, we again have to solve the conservation law~\eqref{eqn: approximate conservation law} with initial condition $u(x,0)=u_\delta(x,t^*)$ where $u_\delta$ denotes the exact solution of~\eqref{eqn: approximate conservation law} up to time $t=t^*$. Note that $u_\delta(\cdot,t^*)$ is a piecewise constant function taking values in the set of breakpoints of $f_\delta$ and the set of values of $u_0^\delta$. In this way, we can determine the unique entropy solution of~\eqref{eqn: approximate conservation law} for all times~\cite{Piccoli/Tournus,holden2015front,dafermos1972polygonal}.

More specifically, we will approximate the flux functions $f^{(i)}$ by piecewise linear flux functions $f^{(i)}_\delta$ by way of interpolating between the points $(j\delta,f^{(i)}(j\delta))$, $j\in\Z$, i.e.,
\begin{equation*}
  f^{(i)}_\delta (u) = f^{(i)}(j\delta) + \frac{f^{(i)}((j+1)\delta) - f^{(i)}(j\delta)}{\delta}(u-j\delta),\qquad u\in (j\delta,(j+1)\delta].
\end{equation*}
Other piecewise constant approximations of $f^{(i)}$ are also possible as long as $\|f^{(i)}-f_\delta^{(i)}\|_{\mathrm{Lip}}\leq C\delta$. For example Baiti and Jenssen \cite{BAITI1997161} interpolate $f^{(i)}$ with uniform spacing in the $f$-values instead of the $u$-values.

The initial datum $u_0$ on the other hand, we approximate by a piecewise constant function $u_0^\delta$ with finitely many discontinuities satisfying
\begin{equation*}
	\TV(u_0^\delta) \leq \TV(u_0)\qquad\text{and}\qquad \|u_0-u_0^\delta\|_{\mathrm{L}^1(\R)} \leq C\delta.
\end{equation*}
As an example, we can take cell averages of $u_0$ over cells of size $\mathcal{O}(\delta)$.





\subsection{Convergence rate estimates for the front tracking method}
The convergence rate estimate for the front tracking method now follows immediately from the stability estimate in \Cref{thm: main stability estimate}.
\begin{theorem}[Convergence rate of the front tracking method]\label{thm: Convergence rate}
	Let $u$ be the entropy solution of~\eqref{conservation law} and $u_\delta$ the front tracking approximation, i.e. the entropy solution of~\eqref{eqn: approximate conservation law}. Then we have the following convergence rate estimate:
	\begin{equation*}
		\|u(\cdot,T) - u_\delta(\cdot,T)\|_{\mathrm{L}^1(\R)} \leq C \delta
	\end{equation*}
	for some constant $C$ independent of $\delta$.
\end{theorem}
\begin{proof}
	Since
	\begin{equation*}
		\|u_0-u_0^\delta\|_{\mathrm{L}^1(\R)} \leq C\delta\qquad\text{and} \qquad \|f^{(i)} - f_\delta^{(i)}\|_{\mathrm{Lip}} \leq \|(f^{(i)})'\|_{\infty} \delta
	\end{equation*}
	and the fact that the front tracking approximation $u_\delta$ is an entropy solution of the conservation law~\eqref{eqn: approximate conservation law} the convergence rate estimate follows immediately from \Cref{thm: main stability estimate}. Since $\TV(u_0^\delta)\leq \TV(u_0)$ and the Lipschitz constant of $f_\delta(k(x),\cdot)$ is bounded by the Lipschitz constant of $f(k(x),\cdot)$ the constant is independent of $\delta$.
\end{proof}

\section{Numerical experiments}\label{sec: numerical experiments}
We present two numerical experiments verifying our convergence rate analysis. We consider the `two flux' case
\begin{gather*}
  \begin{aligned}
    u_t + (H(x)f(u) + (1-H(x))g(u))_x =0,& &&(x,t)\in\R\times(0,T),\\
    u(x,0) = u_0(x),& &&x\in \R
  \end{aligned}
\end{gather*}
where $H$ is the Heaviside function. This corresponds to switching from one $u$-dependent flux, $g$, to another, $f$, across $x=0$.
In order to compare the $\mathrm{L}^1$-error of the front tracking method to that of the finite volume method introduced in~\cite{BadwaikRufconvergence}, we consider the same fluxes and initial data as in~\cite{BadwaikRufconvergence}.
\paragraph{Experiment 1}
In our first numerical experiment we take $g(u) = u$ and $f(u) = \unitfrac{u^2}{2}$ such that we switch from the transport equation to the Burgers equation. The initial datum we consider for Experiment~1 is
\begin{equation*}
  u_0(x) = \begin{cases}
    0.5, &\text{if }x<-0.5,\\
    2, &\text{if }x>-0.5
  \end{cases}
\end{equation*}
which is chosen in such a way that the Rankine--Hugoniot condition at the interface $x=0$ gives $u(0-,t)=u(0+,t)$ before the jump at $x=-0.5$ present in the initial datum interacts with the interface.
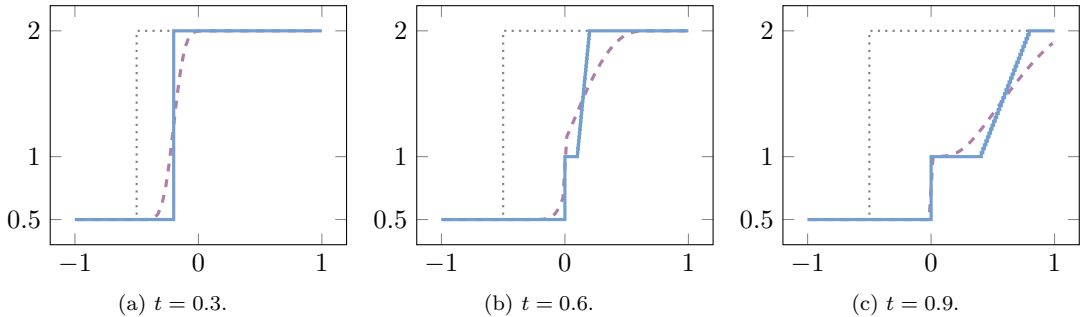
\begin{figure}
\centering
\subfloat[$t=0.3$.]{
\begin{tikzpicture}
  \begin{axis}[ymin=0.3,ymax=2.2,xtick={-1,0,1},xticklabels={$-1$,$0$,$1$},ytick={0.5,1,2},yticklabels={$0.5$,$1$,$2$}]
      \addplot[gray, dotted,thick,mark=none, const plot] table {RiemannID.txt};
      \addplot[plum1,very thick,mark=none, dashed] table {AdvectionToBurgersThird.txt};
    \addplot[skyblue1,very thick,mark=none, const plot] table {SolutionRiemannID_Transport_to_Burgers0.3.txt};
  \end{axis}
\end{tikzpicture}
}
\subfloat[$t=0.6$.]{
\begin{tikzpicture}
  \begin{axis}[ymin=0.3,ymax=2.2,xtick={-1,0,1},xticklabels={$-1$,$0$,$1$},ytick={0.5,1,2},yticklabels={$0.5$,$1$,$2$}]
    \addplot[gray, dotted,thick,mark=none, const plot] table {RiemannID.txt};
    \addplot[plum1,very thick,mark=none, dashed] table {AdvectionToBurgersTwoThird.txt};
    \addplot[skyblue1,very thick,mark=none, const plot] table {SolutionRiemannID_Transport_to_Burgers0.6.txt};
  \end{axis}
\end{tikzpicture}
}
\subfloat[$t=0.9$.]{
\begin{tikzpicture}
  \begin{axis}[ymin=0.3,ymax=2.2,xtick={-1,0,1},xticklabels={$-1$,$0$,$1$},ytick={0.5,1,2},yticklabels={$0.5$,$1$,$2$}]
    \addplot[gray, dotted,thick,mark=none, const plot] table {RiemannID.txt};
    \addplot[plum1,very thick,mark=none,dashed] table {AdvectionToBurgersEnd.txt};
    \addplot[skyblue1,very thick,mark=none, const plot] table {SolutionRiemannID_Transport_to_Burgers0.9.txt};
  \end{axis}
\end{tikzpicture}
}
\caption{Numerical solutions of Experiment $1$ with $\Dx=\unitfrac{2}{64}$ calculated with the front tracking method (straight line) and the finite volume method (dashed line, $\unitfrac{\Dt}{\Dx} = 0.5$) at various times.}
\label{fig: Numerical experiments 1}
\end{figure}
\Cref{fig: Numerical experiments 1} shows the numerical solution computed with the front tracking method with open boundaries in blue (straight line), a numerical solution calculated with the finite volume method~\cite{BadwaikRufconvergence} in purple (dashed line), and the initial datum in gray (dotted line) at various times (before, during and after interaction with the interface). Here, we used $\delta = \Dx = \unitfrac{2}{n}$ with $n=64$ and end time $T=0.9$. We clearly recognize the characteristic features of the transport equation and the Burgers equation here as the upward jump in the initial datum is transported to the right as a shock until it crosses the interface at $x=0$ where the shock, as it enters the Burgers regime, subsequently becomes a rarefaction wave.
\paragraph{Experiment 2}
In our second numerical experiment we choose $g(u)=\unitfrac{u^2}{2}$ and $f(u)=u$ meaning we switch from the Burgers equation to the transport equation across the interface at $x=0$. The initial datum for Experiment 2 is
\begin{equation*}
  u_0(x)=2+\exp(-100(x+0.75)^2).
\end{equation*}
Again, the offset of the initial datum is chosen such that the Rankine--Hugoniot condition at $x=0$ gives $u(0-,t)=u(0+,t)$ before the nonconstant part of $u_0$ interacts with the interface.
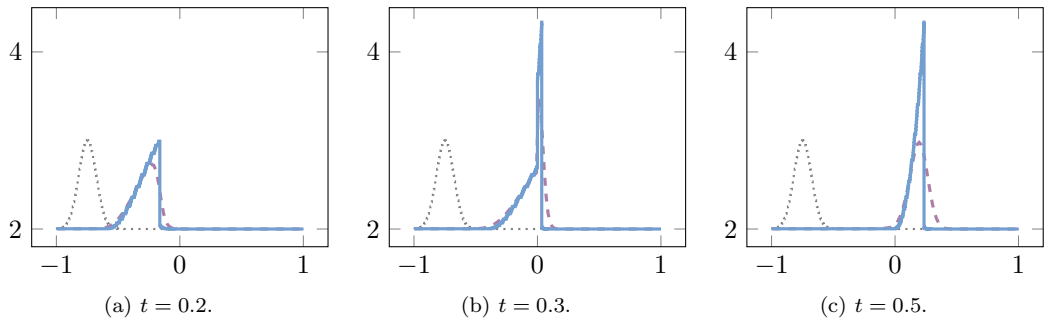
\begin{figure}
\centering
\subfloat[$t= 0.2$.]{
\begin{tikzpicture}
  \begin{axis}[ymin=1.8,ymax=4.5,xtick={-1,0,1},xticklabels={$-1$,$0$,$1$},ytick={2,4},yticklabels={$\phantom{0.}2$,$4$}]
    \addplot[gray, dotted, thick,mark=none, const plot] table {BumpID.txt};
    \addplot[plum1,very thick,mark=none,dashed] table {BurgersToAdvectionHalf.txt};
    \addplot[skyblue1,very thick,mark=none, const plot] table {SolutionBumpID_Burgers_to_Transport0.2.txt};
  \end{axis}
\end{tikzpicture}
}
\subfloat[$t=0.3$.]{
\begin{tikzpicture}
  \begin{axis}[ymin=1.8,ymax=4.5,xtick={-1,0,1},xticklabels={$-1$,$0$,$1$},ytick={2,4},yticklabels={$\phantom{0.}2$,$4$}]
    \addplot[gray, dotted,thick,mark=none, const plot] table {BumpID.txt};
    \addplot[plum1,very thick,mark=none,dashed] table {BurgersToAdvectionTwoThird.txt};
    \addplot[skyblue1,very thick,mark=none, const plot] table {SolutionBumpID_Burgers_to_Transport0.3.txt};
  \end{axis}
\end{tikzpicture}
}
\subfloat[$t=0.5$.]{
\begin{tikzpicture}
  \begin{axis}[ymin=1.8,ymax=4.5,xtick={-1,0,1},xticklabels={$-1$,$0$,$1$},ytick={2,4},yticklabels={$\phantom{0.}2$,$4$}]
    \addplot[gray, dotted,thick,mark=none, const plot] table {BumpID.txt};
    \addplot[plum1,very thick,mark=none, dashed] table {BurgersToAdvectionEnd.txt};
    \addplot[skyblue1,very thick,mark=none, const plot] table {SolutionBumpID_Burgers_to_Transport0.5.txt};
  \end{axis}
\end{tikzpicture}
}
\caption{Numerical solutions of Experiment $2$ with $\Dx=\unitfrac{2}{128}$ calculated with the front tracking method (straight line) and the finite volume method (dashed line, $\unitfrac{\Dt}{\Dx} = 0.2$) at various times.}
\label{fig: Numerical experiments 2}
\end{figure}
\Cref{fig: Numerical experiments 2} shows the numerical solution computed with the front tracking method with open boundaries in blue (straight line), a numerical solution calculated with the finite volume method~\cite{BadwaikRufconvergence} in purple (dashed line) and the initial datum in gray (dotted line) at various times (immediately before, during, and after interaction with the interface). Here, we used $\delta =\Dx=\unitfrac{2}{n}$ with $n=128$ and end time $T=0.5$. We clearly recognize the shock formation due to the Burgers regime to the left of the interface (see \Cref{fig: Numerical experiments 2}(a)) and that the shock is transported across the interface to the right with a different profile due to the Rankine--Hugoniot condition (see \Cref{fig: Numerical experiments 2} (c)).

\begin{table}
  \centering
  \subfloat[Experiment $1$.]{
  \begin{tabular}{rcccc}
  \toprule
  \multicolumn{1}{c}{$n$} & $\Lone$ error & $\Lone$ OOC\\
  \midrule
  $ 16$ &  $\num{6.250e-03}$ &  --  \\
  $ 32$ &  $\num{3.125e-03}$ & $1.00$ \\
  $ 64$ &  $\num{1.562e-03}$ & $1.00$ \\
  $128$ &  $\num{7.813e-04}$ & $1.00$ \\
  $256$ &  $\num{3.906e-04}$ & $1.00$ \\
  $512$ &  $\num{1.953e-04}$ & $1.00$ \\
  $1024$&  $\num{9.766e-05}$ & $1.00$ \\
  \bottomrule
\end{tabular}
  }
  \hspace{2em}
  \subfloat[Experiment $2$.]{
  \begin{tabular}{rcccc}
  \toprule
  \multicolumn{1}{c}{$n$} & $\Lone$ error & $\Lone$ OOC\\
  \midrule
  $ 16$ &  $\num{2.599e-02}$ & -- \\
  $ 32$ &  $\num{1.012e-02}$ & $1.36$ \\
  $ 64$ &  $\num{4.845e-03}$ & $1.06$ \\
  $128$ &  $\num{2.344e-03}$ & $1.05$ \\
  $256$ &  $\num{1.151e-03}$ & $1.03$ \\
  $512$ &  $\num{5.820e-04}$ & $0.98$ \\
  $1024$&  $\num{3.036e-04}$ & $0.94$ \\
  \bottomrule
\end{tabular}
  }
  \caption{$\mathrm{L}^1$ error and observed order of convergence for Experiments $1$ and $2$.}
  \label{tab: Convergence rates}
\end{table}
\Cref{tab: Convergence rates} shows the $\mathrm{L}^1$ errors and observed convergence rates of the front tracking method at time $T=0.9$ for Experiment 1 and at time $T=0.5$ for Experiment 2 for various values of $n$ where $\delta=\Dx=\frac{2}{n}$. As a reference solution, we used a numerical solution on a very fine grid ($n=2048$) in both cases. Both experiments clearly show the first-order convergence rate proved in \Cref{thm: Convergence rate}. Comparing \Cref{tab: Convergence rates} to the corresponding table for the finite volume method~\cite[Tab. 1]{BadwaikRufconvergence}, we see that the $\mathrm{L}^1$ error of the front tracking method seems to be lower by a factor of $100$ for Experiment~1 and by a factor of $10$ for Experiment~2 than the corresponding $\mathrm{L}^1$ error for the finite volume method.

\section{Conclusion}\label{sec: conclusion}

In this paper, we have studied conservation laws with discontinuous flux where we switch from one strictly monotone flux to another across finitely many points in the spatial domain. We extended existing $\mathrm{L}^1$-contractivity results of adapted entropy solutions by showing $\mathrm{L}^1$-Lipschitz-stability with respect to changes not only in the initial datum, but also in the flux $f$ and the discontinuous coefficient $k$. From there, we proved a first-order convergence rate for the front tracking method
which is widely used in the field of conservation laws with discontinuous flux.
We presented numerical experiments substantiating our convergence rate result.
Comparison with finite volume methods indicates better performance of the front tracking method.




\section*{Acknowledgments}
The author wishes to thank Susanne Solem and Espen Sande for their careful reading of the manuscript.

\bibliographystyle{siam}

\end{document}